\documentclass[twoside,leqno,10pt, A4]{amsart}
\usepackage{amsfonts}
\usepackage{amsmath}
\usepackage{amscd}
\usepackage{amssymb}
\usepackage{amsthm}
\usepackage{amsrefs}
\usepackage{latexsym}
\usepackage{mathrsfs}
\usepackage{bbm}
\usepackage{enumerate}
\usepackage{graphicx}

\usepackage{amsfonts}
\usepackage{amsmath}
\usepackage{amscd}
\usepackage{amssymb}
\usepackage{amsthm}
\usepackage{amsrefs}
\usepackage{latexsym}
\usepackage{mathrsfs}
\usepackage{bbm}
\usepackage{amscd}
\usepackage{amssymb}
\usepackage{amsthm}
\usepackage{amsrefs}
\usepackage{latexsym}
\usepackage{mathrsfs}
\usepackage{bbm}
\usepackage{enumerate}
\usepackage{graphicx}
\usepackage{color}
\begin{document}

\newtheorem{theorem}[subsection]{Theorem}
\newtheorem{proposition}[subsection]{Proposition}
\newtheorem{lemma}[subsection]{Lemma}
\newtheorem{corollary}[subsection]{Corollary}
\newtheorem{conjecture}[subsection]{Conjecture}
\newtheorem{prop}[subsection]{Proposition}
\numberwithin{equation}{section}
\newcommand{\mr}{\ensuremath{\mathbb R}}
\newcommand{\mc}{\ensuremath{\mathbb C}}
\newcommand{\dif}{\mathrm{d}}
\newcommand{\intz}{\mathbb{Z}}
\newcommand{\ratq}{\mathbb{Q}}
\newcommand{\natn}{\mathbb{N}}
\newcommand{\comc}{\mathbb{C}}
\newcommand{\rear}{\mathbb{R}}
\newcommand{\prip}{\mathbb{P}}
\newcommand{\uph}{\mathbb{H}}
\newcommand{\fief}{\mathbb{F}}
\newcommand{\majorarc}{\mathfrak{M}}
\newcommand{\minorarc}{\mathfrak{m}}
\newcommand{\sings}{\mathfrak{S}}
\newcommand{\fA}{\ensuremath{\mathfrak A}}
\newcommand{\mn}{\ensuremath{\mathbb N}}
\newcommand{\mq}{\ensuremath{\mathbb Q}}
\newcommand{\half}{\tfrac{1}{2}}
\newcommand{\f}{f\times \chi}
\newcommand{\summ}{\mathop{{\sum}^{\star}}}
\newcommand{\chiq}{\chi \bmod q}
\newcommand{\chidb}{\chi \bmod db}
\newcommand{\chid}{\chi \bmod d}
\newcommand{\sym}{\text{sym}^2}
\newcommand{\hhalf}{\tfrac{1}{2}}
\newcommand{\sumstar}{\sideset{}{^*}\sum}
\newcommand{\sumprime}{\sideset{}{'}\sum}
\newcommand{\sumprimeprime}{\sideset{}{''}\sum}
\newcommand{\sumflat}{\sideset{}{^\flat}\sum}
\newcommand{\shortmod}{\ensuremath{\negthickspace \negthickspace \negthickspace \pmod}}
\newcommand{\V}{V\left(\frac{nm}{q^2}\right)}
\newcommand{\sumi}{\mathop{{\sum}^{\dagger}}}
\newcommand{\mz}{\ensuremath{\mathbb Z}}
\newcommand{\leg}[2]{\left(\frac{#1}{#2}\right)}
\newcommand{\muK}{\mu_{\omega}}
\newcommand{\thalf}{\tfrac12}
\newcommand{\lp}{\left(}
\newcommand{\rp}{\right)}
\newcommand{\Lam}{\Lambda_{[i]}}
\newcommand{\lam}{\lambda}
\def\L{\fracwithdelims}
\def\om{\omega}
\def\pbar{\overline{\psi}}
\def\phis{\phi^*}
\def\lam{\lambda}
\def\lbar{\overline{\lambda}}
\newcommand\Sum{\Cal S}
\def\Lam{\Lambda}
\newcommand{\sumtt}{\underset{(d,2)=1}{{\sum}^*}}
\newcommand{\sumt}{\underset{(d,2)=1}{\sum \nolimits^{*}} \widetilde w\left( \frac dX \right) }

\newcommand{\hf}{\tfrac{1}{2}}
\newcommand{\af}{\mathfrak{a}}
\newcommand{\Wf}{\mathcal{W}}

\theoremstyle{plain}
\newtheorem{conj}{Conjecture}
\newtheorem{remark}[subsection]{Remark}

\makeatletter
\def\widebreve{\mathpalette\wide@breve}
\def\wide@breve#1#2{\sbox\z@{$#1#2$}%
     \mathop{\vbox{\m@th\ialign{##\crcr
\kern0.08em\brevefill#1{0.8\wd\z@}\crcr\noalign{\nointerlineskip}%
                    $\hss#1#2\hss$\crcr}}}\limits}
\def\brevefill#1#2{$\m@th\sbox\tw@{$#1($}%
  \hss\resizebox{#2}{\wd\tw@}{\rotatebox[origin=c]{90}{\upshape(}}\hss$}
\makeatletter

\title[Bounds for moments of modular $L$-functions to a fixed modulus]{Bounds for moments of modular $L$-functions to a fixed modulus}

\author{Peng Gao, Xiaoguang He, and Xiaosheng Wu}
\begin{abstract}
 We study the $2k$-th moment of the family of twisted modular $L$-functions to a fixed prime power modulus at the central values. We establish sharp lower bounds for all real $k \geq 0$ and sharp upper bounds for $k$ in the range $0 \leq k \leq 1$.
\end{abstract}

\maketitle

\noindent {\bf Mathematics Subject Classification (2010)}: 11F66, 11F11  \newline

\noindent {\bf Keywords}: moments, modular $L$-functions, lower bounds, upper bounds

\section{Introduction}
\label{sec 1}

  Evaluating asymptotically moments of $L$-functions at the central point is an important subject in analytical number theory as these moments have many arithmetic applications.  In this paper, we are interested in the families of twisted modular $L$-functions to a fixed modulus. For simplicity, we fix a holomorphic Hecke eigenform $f$ of weight $\kappa $ of level $1$. The Fourier expansion of
  $f$ at infinity can be written as
\[
f(z) = \sum_{n=1}^{\infty} \lambda (n) n^{\frac{\kappa -1}{2}} e(nz),
\]
where we denote $e(z) $ for $e^{2 \pi i z}$. Here the coefficients $\lambda (n)$ are real and satisfy $\lambda (1) =1$ and $|\lambda (n)|
\leq d(n)$ for $n \geq 1$ with $d(n)$ being the divisor function of $n$. We fix a modulus $q$ and assume that $q \not \equiv 2 \pmod 4$ throughout so that primitive Dirichlet characters modulo $q$ exist and denote $\chi$ for a primitive Dirichlet character modulo $q$.
For $\Re(s) > 1$, we define the twisted modular $L$-function $L(s, f \otimes \chi)$ by
\begin{align*}
L(s, f \otimes \chi) &= \sum_{n=1}^{\infty} \frac{\lambda (n)\chi(n)}{n^s}
 = \prod_{p} \left(1 - \frac{\lambda (p) \chi (p)}{p^s}  + \frac{\chi(p^2)}{p^{2s}}\right)^{-1}.
\end{align*}

  In \cite{Stefanicki}, T. Stefanicki obtained asymptotic formulas for the first and second moment of $L(\half, f \otimes \chi)$ over all primitive characters. The result given in \cite{Stefanicki} concerning the second moment is only valid for a restricted values of $q$ and this was later improved by P. Gao, R. Khan and G. Ricotta \cite{GKR} to hold for almost all $q$. Subsequent work with improvements on the error term can be found in \cite{BM15, BFKMM, KMS17}.

  In \cite{BM15}, V. Blomer and D. Mili\'{c}evi\'{c} not only studied the second moments of various modular $L$-functions
  in a much broader sense but also obtained sharp lower bounds for the second moment of mixed product of twisted modular $L$-functions.
  It is mentioned in \cite{BM15} that the method carries over to treat the $2k$-th moment for any integer $k$ as well
  and this was later achieved by G. Chen and X. He \cite{CH}.

  The aim of this paper is to establish sharp lower and upper bounds for the $2k$-th moments of central values of twisted modular $L$-functions to a fixed prime power modulus for real $k$. Throughout the paper, we denote $\phis(q)$ for the number of primitive characters  modulo $q$ and $\sumstar$ for the sum over primitive Dirichlet characters
  modulo $q$.  For lower bounds, we have the following result.
\begin{theorem}
\label{thmlowerbound}
    With notations as above. Let $q=q_0^{\nu}$ where $q_0$ is a large prime number and $\nu \geq 1$. For any real number $k \geq 0$, we have
\begin{align}
\label{lowerbound}
   \sumstar_{\substack{ \chi \shortmod q }}|L(\tfrac{1}{2}, f \otimes \chi)|^{2k} \gg_k \phis(q)(\log q)^{k^2}.
\end{align}
\end{theorem}

   For upper bounds, our result is as follows.
\begin{theorem}
\label{thmupperbound}
    With notations as above. Let $q=q_0^{\nu}$ where $q_0$ is a large prime number and $\nu \geq 1$. For any real number $k$ such that $0 \leq k \leq 1$, we have
\begin{align}
\label{upperbound}
   \sumstar_{\substack{ \chi \shortmod q }}|L(\tfrac{1}{2},f \otimes \chi)|^{2k} \ll_k \phis(q)(\log q)^{k^2}.
\end{align}
\end{theorem}

   Our strategy of proofs for the above results largely follows from the lower bounds principle of W. Heap and K. Soundararajan \cite{H&Sound} as well as the upper bounds principle of M. Radziwi{\l\l} and K. Soundararajan \cite{Radziwill&Sound} concerning moments of general families of $L$-functions. These principles require one to be able to evaluate twisted moments. For our case, we shall employ the method in \cite{Stefanicki} to treat the twisted first moment and apply the tools developed in \cite{BM15} to treat the twisted second moment. Especially, an estimation for certain off-diagonal terms in \cite[Section 12]{BM15} plays a crucial role in our treatment for the twisted second moment.

  Combining Theorem \ref{thmlowerbound} and \ref{thmupperbound}, we have the following result concerning the order of magnitude of our family of $L$-functions.
\begin{theorem}
\label{thmorderofmag}
   With notations as above. Let $q=q_0^{\nu}$ where $q_0$ is a large prime number and $\nu \geq 1$. For any real number $k$ such that $0 \leq k \leq 1$, we have
\begin{align*}
   \sumstar_{\substack{ \chi \shortmod q }}|L(\tfrac{1}{2},\chi)|^{2k} \asymp \phis(q)(\log q)^{k^2}.
\end{align*}
\end{theorem}

   Our results above are in agreement with the conjectured formulas for these moments following from the general philosophy in the work of J. P. Keating and N. C. Snaith \cite{Keating&Snaith2000}.

\section{Preliminaries}
\label{sec 2}

  We reserve the letter $p$ for a prime number in this paper and
we note the following result concerning sums over primes.
\begin{lemma}
\label{RS} With notations as above and let $x \geq 2$.  We have for some constant $b$,
\begin{align}
\label{M1}
\sum_{p\le x} \frac{\lambda^2(p)}{p} = \log \log x + b+ O\Big(\frac{1}{\log x}\Big).
\end{align}
 Also,
\begin{align}
\label{M2}
\sum_{p\le x} \frac {\log p}{p} = \log x + O(1).
\end{align}
\end{lemma}

 The first assertion of Lemma \ref{RS} follows from the
Rankin-Selberg theory for $L(s, f)$,  which can be found in  \cite[Chapter 5]{iwakow}. The second assertion of Lemma \ref{RS}
is given in \cite[Lemma 2.7]{MVa1}.

  Next, we note the following approximate functional equations for $ L(\half, f \otimes \chi)$ and $|L(1/2, f \otimes \chi)|^2$.
\begin{lemma}
\label{PropDirpoly}
   For $X>0$, we have
\begin{align}
\label{lstapprox}
 L(\half, f \otimes \chi) = \sum^{\infty}_{n=1} \frac{\lambda(n)\chi(n)}{\sqrt{n}} W\left(\frac {nX}{q}\right)+\iota_{\chi}\sum^{\infty}_{n=1} \frac{\lambda(n)\overline \chi(n)}{\sqrt{n}} W\left(\frac {n}{qX}\right),
\end{align}
  where $\iota_{\chi}=i^{\kappa}\tau(\chi)^2/q$, $\tau(\chi)$ is the Gauss sum associated to $\chi$ and

$$ W(x) = \frac{1}{2\pi i} \int\limits_{(c)} \frac{\Gamma (\frac{\kappa}{2} +s)}{\Gamma (\frac{\kappa}{2})} e^{s^2} (2\pi x)^{-s} \> \frac{ds}{s}.$$
  We also have
\begin{align}
\label{lsquareapprox}
|L(\half, f \otimes \chi)|^2 = 2 \sum^{\infty}_{a, b=1} \frac{\chi(a) \overline{\chi}(b)}{\sqrt{ab}} \Wf \left(\frac {ab}{q^2}\right),
\end{align}
  where
$$ \Wf (x) = \frac{1}{2\pi i} \int\limits_{(c)} \frac{\Gamma\left(\frac{\kappa}{2} + s \right)^2}{(2\pi)^s \Gamma\left(\frac{\kappa}{2} \right)^2 } x^{-s} \> \frac{ds}{s}.$$
  Moreover, the functions $w(x), \Wf (x)$ are real valued and satisfy the bound that for any $c>0$,
\begin{align}
\label{W}
 W(x), \Wf (x)  \ll_c \min( 1 , x^{-c}).
\end{align}
\end{lemma}

    The functional equation given in \eqref{lstapprox} can be derived using standard arguments as those in the proof of \cite[Theorem 5.3]{iwakow}. The functional equation given in \eqref{lsquareapprox} can be found in \cite[Lemma 2.2]{GKR}.

\section{Outline of the Proofs}
\label{sec 2'}

   As the case $k=0$ is trivial and the case $k=1$ for
  \eqref{lowerbound}-\eqref{upperbound} is known from \cite{BM15}, we may assume in our proofs that $0< k \neq 1$ is a fixed positive real number.
We let $N, M$ be two large natural numbers depending on $k$ only and define a sequence of even natural
  numbers $\{ \ell_j \}_{1 \leq j \leq R}$ such that $\ell_1= 2\lceil N \log \log q\rceil$ and $\ell_{j+1} = 2 \lceil N \log \ell_j \rceil$ for
  $j \geq 1$, where we set $R$ to be the largest natural number satisfying $\ell_R >10^M$.  We shall choose $M$ large enough so that we have $\ell_{j} >
  \ell_{j+1}^2$ for all $1 \leq j \leq R-1$ and we note that this implies that
\begin{align}
\label{sumoverell}
  R \ll \log \log \ell_1, \quad \sum^R_{j=1}\frac 1{\ell_j} \leq \frac 2{\ell_R}.
\end{align}

    We define ${ P}_1$ to be the set of odd primes not exceeding $q^{1/\ell_1^2}$ and
${P_j}$ to be the set of primes lying in the interval $(q^{1/\ell_{j-1}^2}, q^{1/\ell_j^2}]$ for $2\le j\le R$. For each $1 \leq j \leq R$, we
write
\begin{equation*}
{\mathcal P}_j(\chi) = \sum_{p\in P_j} \frac{1}{\sqrt{p}} \chi(p), \quad  {\mathcal Q}_j(\chi, k) =\Big (\frac{c_k {\mathcal
P}_j(\chi) }{\ell_j}\Big)^{r_k\ell_j},
\end{equation*}
  where we define
\begin{align}
\label{r_k}
\begin{split}
  c_k= & 64 \max (1, k), \\
r_k =& \left\{
 \begin{array}
  [c]{ll}
  \lceil 1+1/k \rceil+1 & k>1,\\
  \lceil k /(2k-1) \rceil+1 & k<1.
 \end{array}
 \right.
 \end{split}
\end{align}
  We also define ${\mathcal Q}_{R+1}(\chi, k)=1$.

  For any non-negative integer $\ell$ and any real number $x$, we denote
\begin{equation*}
E_{\ell}(x) = \sum_{j=0}^{\ell} \frac{x^{j}}{j!}.
\end{equation*}
  We further define for each $1 \leq j \leq R$ and any real number $\alpha$,
\begin{align*}
{\mathcal N}_j(\chi, \alpha) = E_{\ell_j} (\alpha {\mathcal P}_j(\chi)), \quad \mathcal{N}(\chi, \alpha) = \prod_{j=1}^{R} {\mathcal
N}_j(\chi,\alpha).
\end{align*}

 We point out here that in the rest of the paper, when we use
 $\ll$ or the $O$-symbol to estimate various quantities presented, the implicit constants involved depend on $k$ only and are uniform with
 respect to $p, \chi$. We shall also make the convention that an empty product is defined to be $1$.

  The following two lemmas present in our setting the lower bounds principle of W. Heap and K. Soundararajan in \cite{H&Sound} and the
  upper bounds principle of M. Radziwi{\l\l} and K. Soundararajan in \cite{Radziwill&Sound}, respectively.

  The first one corresponds to the lower bounds principle.
\begin{lemma}
\label{lem1}
 With notations as above. For $0<k<1$, we have
\begin{align}
\label{basiclowerbound}
\begin{split}
\sumstar_{\substack{ \chi \shortmod q }}L(\tfrac{1}{2},f \otimes \chi)  \mathcal{N}(\chi, k-1) \mathcal{N}(\overline{\chi}, k)
 \ll & \Big ( \sumstar_{\substack{ \chi \shortmod q }}|L(\tfrac{1}{2}, f \otimes \chi)|^{2k} \Big )^{1/2}\Big ( \sumstar_{\substack{ \chi \shortmod q
 }}|L(\tfrac{1}{2}, f \otimes  \chi)|^2 |\mathcal{N}(\chi, k-1)|^2  \Big)^{(1-k)/2} \\
 & \times \Big ( \sumstar_{\substack{ \chi \shortmod q }}   \prod^R_{j=1}\big ( |{\mathcal N}_j(\chi, k)|^2+ |{\mathcal Q}_j(\chi,k)|^2 \big )
 \Big)^{k/2}.
\end{split}
\end{align}
 For $k>1$, we have
\begin{align}
\label{basicboundkbig}
\begin{split}
 & \sumstar_{\substack{ \chi \shortmod q }}L(\tfrac{1}{2}, f \otimes \chi)  \mathcal{N}(\chi, k-1) \mathcal{N}(\overline{\chi}, k)
 \ll   \Big ( \sumstar_{\substack{ \chi \shortmod q }}|L(\tfrac{1}{2},f \otimes  \chi)|^{2k} \Big )^{\frac {1}{2k}}\Big ( \sumstar_{\substack{ \chi
 \shortmod q }} \prod^R_{j=1} \big ( |{\mathcal N}_j(\chi, k)|^2+ |{\mathcal Q}_j(\chi,k)|^2 \big ) \Big)^{\frac {2k-1}{2k}}.
\end{split}
\end{align}
  The implied constants in \eqref{basiclowerbound} and \eqref{basicboundkbig} depend on $k$ only.
\end{lemma}
\begin{proof}
   We consider the case $0<k<1$ first and we apply H\"older's inequality to see that the left side of \eqref{basiclowerbound} is
\begin{align*}
\begin{split}
 \leq & \Big ( \sumstar_{\substack{ \chi \shortmod q }}|L(\tfrac{1}{2},f \otimes  \chi)|^{2k} \Big )^{1/2}\Big ( \sumstar_{\substack{ \chi \shortmod q
 }}|L(\tfrac{1}{2},f \otimes \chi) \mathcal{N}(\chi, k-1)|^2   \Big)^{(1-k)/2}\Big ( \sumstar_{\substack{ \chi \shortmod q }} |\mathcal{N}(\chi,
 k)|^{2/k}|\mathcal{N}(\chi, k-1)|^{2}  \Big)^{k/2}.
\end{split}
\end{align*}

    As in the proof of \cite[Lemma 3.4]{Gao2021-3}, we have for $|z| \le aK/20$ with $0<a \leq 2$,
\begin{align}
\label{Ebound}
\Big| \sum_{r=0}^K \frac{z^r}{r!} - e^z \Big| \le \frac{|z|^{K}}{K!} \le \Big(\frac{a e}{20}\Big)^{K}.
\end{align}
  Setting $z=\alpha {\mathcal P}_j(\chi), K=\ell_j $ and $a=\min (|\alpha|, 2 )$ in \eqref{Ebound} then implies that when
  $|{\mathcal P}_j(\chi)| \le \ell_j/(20(1+|\alpha|))$,
\begin{align*}
{\mathcal N}_j(\chi, \alpha) \leq & \exp ( \alpha{\mathcal P}_j(\chi) )\left( 1+ \exp ( |\alpha {\mathcal P}_j(\chi)| ) \left( \frac{a e}{20} \right)^{ \ell_j}  \right) \leq \exp ( \alpha {\mathcal P}_j(\chi)  ) \left( 1+   e^{-\ell_j}  \right).
\end{align*}

  Similarly, we have
\begin{align*}
{\mathcal N}_j(\chi, \alpha) \geq & \exp ( \alpha{\mathcal P}_j(\chi) )\left( 1- \exp ( |\alpha {\mathcal P}_j(\chi)| ) \left( \frac{a e}{20} \right)^{ \ell_j} \right) \geq \exp ( \alpha {\mathcal P}_j(\chi)  ) \left( 1-   e^{-\ell_j}  \right).
\end{align*}

 We apply the above estimations to ${\mathcal N}_j(\chi, k-1), {\mathcal N}_j(\chi, k)$ to see that when $0<k<1$ and $|{\mathcal P}_j(\chi)| \le
 \ell_j/60$, then
\begin{align*}
\begin{split}
|\mathcal{N}_j(\chi,
 k-1)|^{2}
\leq & \exp ( 2(k-1) \Re {\mathcal P}_j(\chi)  ) \left( 1+  e^{-\ell_j}  \right)^{2} \\
 \leq & |{\mathcal N}_j(\chi, k)|^{\frac {2(k-1)}{k}} \left( 1+e^{-\ell_j} \right )^{2}\left( 1-e^{-\ell_j} \right )^{-\frac {2(k-1)}{k}}.
\end{split}
\end{align*}

 We then conclude that when $|{\mathcal P}_j(\chi)| \le \ell_j/60$,
\begin{align}
\label{est1}
|{\mathcal N}_j(\chi, k)^{\frac {1}{k}} {\mathcal N}_j(\chi, k-1)|^{2}
\leq & |{\mathcal N}_j(\chi, k)|^2 \left( 1+e^{-\ell_j} \right )^{2}\left( 1-e^{-\ell_j} \right )^{-\frac {2(k-1)}{k}}.
\end{align}

  On the other hand, we notice that when $|{\mathcal P}_j(\chi)| > \ell_j/60$,
\begin{align*}
\begin{split}
\max \Big (|{\mathcal N}_j(\chi, k-1)|, |{\mathcal N}_j(\chi, k)| \Big ) &\le \sum_{r=0}^{\ell_j} \frac{|{\mathcal P}_j(\chi)|^r}{r!} \le
|{\mathcal P}_j(\chi)|^{\ell_j} \sum_{r=0}^{\ell_j} \Big( \frac{60}{\ell_j}\Big)^{\ell_j-r} \frac{1}{r!}   \le \Big( \frac{64 |{\mathcal
P}_j(\chi)|}{\ell_j}\Big)^{\ell_j} .
\end{split}
\end{align*}
  It follows that when $|{\mathcal P}_j(\chi)| > \ell_j/60$, we have
\begin{align}
\label{est2}
|{\mathcal N}_j(\chi, k)^{\frac {1}{k}} {\mathcal N}_j(\chi, k-1)|^{2}
& \leq \Big( \frac{64 |{\mathcal P}_j(\chi)|}{\ell_j}\Big)^{2(1+1/k)\ell_j} \leq  |{\mathcal Q}_j(\chi, k)|^2.
\end{align}

 We apply \eqref{est1} and \eqref{est2} to deduce that when $0<k<1$, we have
\begin{align*}
\begin{split}
 \sumstar_{\substack{ \chi \shortmod q }} |\mathcal{N}(\chi,
 k)|^{2/k}|\mathcal{N}(\chi, k-1)|^{2} \le &
\sumstar_{\substack{ \chi \shortmod q }}   \Big ( \prod^R_{j=1} \Big (|{\mathcal N}_j(\chi, k)|^2 \left( 1+e^{-\ell_j} \right )^{2}\left( 1-e^{-\ell_j} \right )^{-\frac {2(k-1)}{k}}+   |{\mathcal Q}_j(\chi, k)|^2  \Big )\Big
) \\
\leq & \prod^R_{j=1} \max \Big( \left( 1+e^{-\ell_j} \right )^{2}\left( 1-e^{-\ell_j} \right )^{-\frac {2(k-1)}{k}}, 1 \Big ) \sumstar_{\substack{ \chi \shortmod q }}   \prod^R_{j=1} \Big (|{\mathcal N}_j(\chi, k)|^2 +  |{\mathcal Q}_j(\chi, k)|^2  \Big )  \\
\ll & \sumstar_{\substack{ \chi \shortmod q }} \prod^R_{j=1} \Big ( |{\mathcal N}_j(\chi, k)|^2 +  |{\mathcal Q}_j(\chi, k)|^2   \Big ),
\end{split}
\end{align*}
 where the last estimation above follows by noting that
\begin{align*}
\begin{split}
 \prod^R_{j=1} \max \Big( \left( 1+e^{-\ell_j} \right )^{2}\left( 1-e^{-\ell_j} \right )^{-\frac {2(k-1)}{k}}, 1 \Big ) \ll 1.
\end{split}
\end{align*}

 Next, we consider the case $k>1$ and we apply H\"older's inequality again to see that the left side of \eqref{basicboundkbig} is
\begin{align}
\label{holderkbig}
\begin{split}
 \leq  \Big ( \sumstar_{\substack{ \chi \shortmod q }}|L(\tfrac{1}{2}, f \otimes \chi)|^{2k} \Big )^{\frac {1}{2k}}\Big ( \sumstar_{\substack{ \chi
 \shortmod q }} |\mathcal{N}(\chi, k)\mathcal{N}(\chi, k-1)|^{\frac {2k}{2k-1}}  \Big)^{\frac {2k-1}{2k}}.
\end{split}
\end{align}

  We argue similar to above to see that when  $|{\mathcal P}_j(\chi)| \le \ell_j/(40k)$,
\begin{align}
\label{prodNkbig}
\begin{split}
 |\mathcal{N}_j(\chi, k)\mathcal{N}_j(\chi, k-1)|^{\frac {2k}{2k-1}} \leq |{\mathcal N}_j(\chi, k)|^2 \Big( 1+ e^{-\ell_j}  \Big)^{\frac {2k}{2k-1}}\left( 1-e^{-\ell_j} \right )^{-\frac {2(k-1)}{2k-1}}.
\end{split}
\end{align}
  Similarly, when $|{\mathcal P}_j(\chi)| > \ell_j/(40k)$, we have
\begin{align*}
\begin{split}
 |\mathcal{N}_j(\chi, k)\mathcal{N}_j(\chi, k-1)|^{\frac {2k}{2k-1}}
  \leq  |{\mathcal Q}_j(\chi, k)|^2.
\end{split}
\end{align*}
  Combining \eqref{holderkbig}, \eqref{prodNkbig} and arguing as above, we readily deduce the estimation given in \eqref{basicboundkbig}.
  This completes the proof of the lemma.
\end{proof}

  Our next lemma corresponds to the upper bounds principle.
\begin{lemma}
\label{lem2}
 With notations as above. We have for $0<k<1$,
\begin{align}
\label{basiclowerbound2}
\begin{split}
& \sumstar_{\substack{ \chi \shortmod q }}|L(\tfrac{1}{2},f \otimes \chi)|^{2k} \\
 \ll & \Big ( \sumstar_{\substack{ \chi \shortmod q }}|L(\tfrac{1}{2},f \otimes \chi)|^2 \sum^{R}_{v=0} \Big (\prod^v_{j=1} |\mathcal{N}_j(\chi, k-1)|^2 \Big ) |{\mathcal
 Q}_{v+1}(\chi, k)|^{2}
 \Big)^{k} \Big ( \sumstar_{\substack{ \chi \shortmod q }}  \sum^{R}_{v=0}  \Big (\prod^v_{j=1}|\mathcal{N}_j(\chi, k)|^2\Big )|{\mathcal
 Q}_{v+1}(\chi, k)|^{2} \Big)^{1-k},
\end{split}
\end{align}
  where the implied constants depend on $k$ only.
\end{lemma}
\begin{proof}
  Using arguments similar to those in the proof of Lemma \ref{lem1}, we see that when $|{\mathcal P}_j(\chi)| \le \ell_j/60$,
\begin{align}
\label{prodNlowerbound}
 |\mathcal{N}_j(\chi, k-1)|^{2k}|\mathcal{N}_j(\chi, k)|^{2(1-k)} \geq \big(1- e^{-\ell_j} \big )^2.
\end{align}

   If there exists an integer $0 \leq v \leq R-1$ such that $| \mathcal{P}_j (\chi) | \leq \ell_j/60$ for all $j \leq v$ and that $|
  \mathcal{P}_{v+1} (\chi) | > \ell_{v+1}/60$,  we deduce from the above and the observation that $|{\mathcal Q}_{v+1}(\chi, k)| \geq 1$ when
  $|{\mathcal P}_{v+1}(\chi)| \ge \ell_{v+1}/60$ that
\begin{align*}
 \Big ( \prod^v_{j=1}|\mathcal{N}_j(\chi, k-1)|^{2k}|\mathcal{N}_j(\chi, k)|^{2(1-k)} \Big )|{\mathcal Q}_{v+1}(\chi, k)|^2 \gg 1.
\end{align*}

  If no such $v$ exists, then $| \mathcal{P}_j (\chi) | \leq \ell_j/60$ for all $1 \leq j \leq R$. Thus the estimation
  \eqref{prodNlowerbound} holds for all $j$ and we have
\begin{align*}
 \prod^R_{j=1}|\mathcal{N}_j(\chi, k-1)|^{2k}|\mathcal{N}_j(\chi, k)|^{2(1-k)} \gg 1.
\end{align*}

  In either case, we conclude that
\begin{align*}
 \Big (\sum^R_{v=0}\Big ( \prod^v_{j=1}|\mathcal{N}_j(\chi, k-1)|^2 \Big )|{\mathcal Q}_{v+1}(\chi, k)|^2 \Big )^{k}\Big (\sum^R_{v=0}\Big ( \prod^v_{j=1}|\mathcal{N}_j(\chi, k)|^2\Big ) |{\mathcal Q}_{v+1}(\chi, k)|^2 \Big )^{1-k}  \gg 1.
\end{align*}
   It follows from this that we have
\begin{align*}
 & \sumstar_{\substack{ \chi \shortmod q }} |L(\half, f \otimes \chi)|^{2k} \\
 \ll & \sumstar_{\substack{ \chi \shortmod q }} |L(\half, f \otimes \chi)|^{2k} \Big (\sum^R_{v=0}\Big ( \prod^v_{j=1}|\mathcal{N}_j(\chi, k-1)|^2 \Big )|{\mathcal Q}_{v+1}(\chi, k)|^2 \Big )^{k}
 \times \Big (\sum^R_{v=0}\Big (  \prod^v_{j=1}|\mathcal{N}_j(\chi, k)|^2\Big ) |{\mathcal Q}_{v+1}(\chi, k)|^2 \Big )^{1-k}.
\end{align*}
     Applying H\"older's inequality to the last expression above leads to the estimation given in \eqref{basiclowerbound2} and this completes the proof of the lemma.
\end{proof}

  We apply Lemma \ref{lem1} and Lemma \ref{lem2} to see that in order to prove Theorem \ref{thmlowerbound} and Theorem
   \ref{thmupperbound}, it suffices to establish the following three propositions.
\begin{proposition}
\label{Prop4} With notations as above. We have for $k>0$,
\begin{align*}
\sumstar_{\substack{ \chi \shortmod q }}L(\tfrac{1}{2}, f \otimes \chi) \mathcal{N}(\overline{\chi}, k) \mathcal{N}(\chi, k-1) \gg \phis(q)(\log q)^{ k^2
} .
\end{align*}
\end{proposition}

\begin{proposition}
\label{Prop5} With notations as above. We have for $0<k < 1$,
\begin{align*}
\max \Big (  \sumstar_{\substack{ \chi \shortmod q }}|L(\tfrac{1}{2},f \otimes \chi)\mathcal{N}(\chi, k-1)|^2, \sumstar_{\substack{ \chi \shortmod q }}|L(\tfrac{1}{2},f \otimes \chi)|^2 \sum^{R}_{v=0}\Big (\prod^v_{j=1}|\mathcal{N}_j(\chi, k-1)|^{2}\Big ) |{\mathcal
 Q}_{v+1}(\chi, k)|^2 \Big )   \ll
\phis(q)(\log q)^{ k^2 }.
\end{align*}
\end{proposition}

\begin{proposition}
\label{Prop6} With notations as above. We have for $k>0$,
\begin{align*}
\max \Big ( \sumstar_{\substack{ \chi \shortmod q }}\prod^R_{j=1}\big ( |{\mathcal N}_j(\chi, k)|^2+ |{\mathcal Q}_j(\chi,k)|^2 \big ),  \sumstar_{\substack{ \chi \shortmod q }} \sum^{R}_{v=0} \Big ( \prod^v_{j=1}|\mathcal{N}_j(\chi, k)|^{2}\Big )|{\mathcal
 Q}_{v+1}(\chi, k)|^2 \Big )   \ll \phis(q)(\log q)^{ k^2 }.
\end{align*}
\end{proposition}

   We shall prove the above propositions in the rest of the paper.

\section{Proof of Proposition \ref{Prop4}}
\label{sec 4}

    We define $\widetilde{\lambda}(n)$ to be the completely multiplicative function $\widetilde{\lambda}(p)=\lambda(p)$ on primes $p$, we also
    define $w(n)$ to be the multiplicative function such that
    $w(p^{\alpha}) = \alpha!$ for prime powers $p^{\alpha}$.
    We further denote $\Omega(n)$ for the number of prime powers dividing $n$ and let $b_j(n), 1 \leq j \leq R$ be functions such that $b_j(n)$ only takes
    values $0$ or $1$ and $b_j(n)=1$ if and only if $\Omega(n) \leq \ell_j$ and the primes dividing $n$ are all from the interval $P_j$.
    We use these notations to write ${\mathcal N}_j(\chi, \alpha)$ as
\begin{equation}
\label{5.1}
{\mathcal N}_j(\chi, \alpha) = \sum_{n_j} \frac{\widetilde{\lambda}(n_j)}{\sqrt{n_j}} \frac{\alpha^{\Omega(n_j)}}{w(n_j)}  b_j(n_j) \chi(n_j), \quad 1\le j\le R.
\end{equation}
    Note that each ${\mathcal N}_j(\chi, \alpha)$ is a short Dirichlet polynomial since $b_j(n_j)=0$ unless $n_j \leq
    (q^{1/\ell_j^2})^{\ell_j}=q^{1/\ell_j}$. We then deduce that both ${\mathcal N}(\chi, k)$ and ${\mathcal N}(\chi, k-1)$ are short Dirichlet
    polynomials whose lengths are both at most $q^{1/\ell_1+ \ldots +1/\ell_R} < q^{2/10^{M}}$ by \eqref{sumoverell}.
    Note further that we have $\widetilde{\lambda}_f(n_j) \leq 2^{\Omega(n_j)}$ since $\lambda_f(p) \leq 2$ and $b_j(n_j)$ restricts $n_j$
     to satisfy $\Omega(n_j) \leq \ell_j$. This observation allows us to also write ${\mathcal N}_j(\chi, \alpha)$ as
\begin{align}
\label{Nj}
  {\mathcal N}_j(\chi, \alpha)=\sum_{n_j \leq q^{1/\ell_j}}\frac{c_{n_j}b_j(n_j)}{\sqrt{n_j}}\chi(n_j),
\end{align}
   where we have, for some constant $B_0(\alpha)$ depending on $\alpha$ only,
\begin{align*}
  |c_{n_j}| \leq B_0(\alpha)^{\ell_j}.
\end{align*}

  We apply \eqref{Nj} to write for simplicity that
\begin{align*}
 {\mathcal N}(\chi, k-1)= \sum_{a  \leq q^{2/10^{M}}} \frac{x_a}{\sqrt{a}} \chi(a), \quad \mathcal{N}(\overline{\chi}, k) = \sum_{b  \leq
 q^{2/10^{M}}} \frac{y_b}{\sqrt{b}}\overline{\chi}(b),
\end{align*}
    where for some constant $B(k)$ depending on $k$ only, we have
\begin{align}
\label{xybounds}
 x_a, y_b \ll B(k)^{\sum^R_{j=1}\ell_j}  \ll B(k)^{R\ell_1} \ll q^{\varepsilon}.
\end{align}
  The last estimation above follows from \eqref{sumoverell}.

  We now deduce from Lemma \ref{PropDirpoly} that
\begin{align}
\label{Lfirstmoment}
\begin{split}
& \sumstar_{\substack{ \chi \shortmod q }}L(\tfrac{1}{2},f \otimes \chi) \mathcal{N}(\overline{\chi}, k) \mathcal{N}(\chi, k-1) \\
= & \sumstar_{\substack{ \chi \shortmod q }}\sum_{m}\frac {\lambda(m)\chi(m)}{\sqrt{m}}\mathcal{N}(\overline{\chi}, k) \mathcal{N}(\chi,
k-1) W\left(\frac {mX}{q}\right)+\sumstar_{\substack{ \chi \shortmod q }}\iota_{\chi} \sum_{m}\frac {\lambda(m)\overline\chi(m)}{\sqrt{m}}\mathcal{N}(\overline{\chi}, k) \mathcal{N}(\chi,
k-1) W\left(\frac {m}{qX}\right).
\end{split}
\end{align}

  We denote $\mu$ for the M\"obius function and note the following relation (see \cite[(2.3)]{Stefanicki})
\begin{align}
\label{sumchistar}
 & \sumstar_{\substack{ \chi \shortmod q }}\chi(a)=\sum_{c | (q, a-1)}\mu(q/c)\phi(c).
\end{align}
   In particular, setting $a=1$ above implies that
\begin{align}
\label{chistar}
 & \phis(q)=\sum_{c | q}\mu(q/c)\phi(c).
\end{align}

  We apply \eqref{sumchistar} to see that the right side expression in \eqref{Lfirstmoment} equals to
\begin{align*}
 \sum_{c | q}\mu(q/c)\phi(c) \sum_{a} \sum_{b} \sum_{\substack{am \equiv b \bmod c}}\frac {\lambda(m) x_a y_b}{\sqrt{abm}}W\left(\frac {mX}{q}\right)+\sum_{a} \sum_{b} \sum_{\substack{m}}\frac {\lambda(m)x_ay_b}{\sqrt{abm}}W\left(\frac {m}{qX}\right)\sumstar_{\substack{ \chi \shortmod q }}\iota_{\chi}\chi(a)\overline \chi(mb).
\end{align*}

  We evaluate the last summation above by applying the definition of $\iota_{\chi}$ to see that
\begin{align*}
 \sumstar_{\substack{ \chi \shortmod q }}\iota_{\chi}\chi(a)\overline \chi(mb) = \frac {i^{\kappa}}{q}  \sumstar_{\substack{ \chi \shortmod q }} \chi(a)\overline \chi(mb)  \sumstar_{\substack{ v \shortmod q }} \chi(v) S(1,v,q)=\frac {i^{\kappa}}{q}\sum_{c | q}\mu(q/c)\phi(c)\sumstar_{\substack{ v \shortmod q \\ av \equiv mb \shortmod c} } S(1,v,q),
\end{align*}
  where $S$ is the Kloosterman's sum defined by
\begin{align*}
  S(u,v, q)=\sumstar_{\substack{ h \shortmod q }} \exp (\frac {uh+v\overline h}{q}).
\end{align*}

   It follows from the well-known Weil's bound for Kloosterman's sum (see \cite[Corollary 11.20]{iwakow}) that we have
\begin{align*}
  |S(1,v, q)| \leq d(q)q^{1/2}.
\end{align*}

   We then deduce from this that
\begin{align*}
 \sumstar_{\substack{ \chi \shortmod q }}\iota_{\chi}\chi(a)\overline \chi(mb) \leq \frac {1}{q}\sum_{c | q}\phi(c)\frac {q}{c} d(q)q^{1/2} \ll q^{\half+\varepsilon}.
\end{align*}

  The above then implies that we have
\begin{align}
\label{Firstmomentsum2}
& \sum_{a} \sum_{b} \sum_{\substack{m}}\frac {\lambda(m)x_ay_b}{\sqrt{abm}}W\left(\frac {m}{qX}\right)\sumstar_{\substack{ \chi \shortmod q }}\iota_{\chi}\chi(a)\overline \chi(mb)
 \ll  q^{4/10^{M}}q^{1/2+\varepsilon}X^{\varepsilon}\sqrt{qX}.
\end{align}

  It remains to evaluate
\begin{align}
\label{Firstmomentsum1}
 \sum_{c | q}\mu(q/c)\phi(c) \sum_{a} \sum_{b} \sum_{\substack{am \equiv b \bmod c}}\frac {\lambda(m) x_a y_b}{\sqrt{abm}}W\left(\frac {mX}{q}\right).
\end{align}

  We first consider the contribution from the terms $am=b+l c$ with $l \geq 1$ above (note that we may take $M$ large enough
  so that we can not have $b > am$ in our case). By the rapid decay of $W(x)$ given in \eqref{W}, we may assume that $m \leq (q/X)^{1+\varepsilon}$. This then implies that $l \leq  q^{1+2/10^{M}+\varepsilon}/(Xc)$, so that we deduce together with the observation that $x_a, y_b \ll q^{\varepsilon}$ that the total
  contribution from these terms is
\begin{align}
\label{Firstmomentsum1nondiag1}
 \ll &  \sum_{c | q} \phi(c)  q^{\varepsilon}X^{\varepsilon} \sum_{b  \leq q^{2/10^{M}}}  \sum_{l \leq q^{1+2/10^{M}+\varepsilon}/(Xc)}\frac {d(b+l c)}{\sqrt{bdl}} \ll X^{-1/2+\varepsilon}q^{\half+2/10^{M}+\varepsilon}.
\end{align}

  We now set $X=q^{-16/10^{M}}$ to see from \eqref{Lfirstmoment}, \eqref{chistar}-\eqref{Firstmomentsum1nondiag1} that we have
\begin{align*}
& \sumstar_{\substack{ \chi \shortmod q }}L(\tfrac{1}{2},f \otimes \chi) \mathcal{N}(\overline{\chi}, k) \mathcal{N}(\chi, k-1)
\gg \phis(q) \sum_{a} \sum_{b} \sum_{\substack{m \leq (q/X)^{1+\varepsilon} \\ am = b }}\frac {\lambda(m) x_a y_b}{\sqrt{abm}}
= \phis(q) \sum_{b} \frac {y_b}{b} \sum_{\substack{a, m \\ am = b }}\lambda(m)x_a,
\end{align*}
  where the last equality above follows from the observation that $b \leq q^{2/10^{M}}<(q/X)^{1+\varepsilon}$.

   Notice that
\begin{align}
\label{sumab}
\begin{split}
\sum_{b} \frac {y_b}{b} \sum_{\substack{a | b }}\lambda(m)x_a=& \prod^R_{j=1}\Big ( \sum_{n_j} \frac{\widetilde{\lambda}(n_j)}{n_j} \frac{k^{\Omega(n_j)}}{w(n_j)}
b_j(n_j)\sum_{n'_j|n_j} \frac{\widetilde{\lambda}(n'_j)(k-1)^{\Omega(n'_j)}}{w(n'_j)}  b_j(n'_j)\lambda(n_j/n'_j) \Big ) \\
=& \prod^R_{j=1}\Big ( \sum_{n_j} \frac{\widetilde{\lambda}(n_j)}{n_j} \frac{k^{\Omega(n_j)}}{w(n_j)}  b_j(n_j)\sum_{n'_j|n_j}  \frac{\widetilde{\lambda}(n'_j)(k-1)^{\Omega(n'_j)}}{w(n'_j)}\lambda(n_j/n'_j)
\Big ),
\end{split}
\end{align}
  where the last equality above follows by noting that $b_j(n_j)=1$ implies that $b_j(n'_j)=1$ for all $n'_j|n_j$.

  We consider the sum above over $n_j$ for a fixed $1 \leq j \leq R$ in \eqref{sumab}. Note that the factor $b_j(n_j)$ restricts $n_j$ to have
  all prime factors in $P_j$ such that $\Omega(n_j) \leq \ell_j$. If we remove the restriction on $\Omega(n_j)$, then the sum becomes
\begin{align}
\label{6.02}
\begin{split}
& \prod_{\substack{p\in P_j }} \Big( \sum_{i=0}^{\infty} \frac{\lambda^i(p)}{p^i} \frac{k^{i}}{i!}\Big ( \sum_{l=0}^{i}
\frac{(k-1)^{l}}{l!} \lambda^l(p)\lambda(p^{i-l})\Big ) \Big)
= \prod_{\substack{p\in P_j }}\Big (1+ \frac {k^2\lambda^2(p)}p+O(\frac 1{p^2}) \Big ).
\end{split}
\end{align}

   On the other hand, using Rankin's trick by noticing that $2^{\Omega(n_j)-\ell_j}\ge 1$ if $\Omega(n_j) > \ell_j$,  we see that the error
   introduced this way does not exceed
\begin{align*}
\begin{split}
 & \sum_{n_j} \frac{\widetilde{\lambda}(n_j)}{n_j} \frac{k^{\Omega(n_j)}}{w(n_j)}2^{\Omega(n_j)-\ell_j}  \sum_{n'_j|n_j}  \frac{\widetilde{\lambda}(n'_j)|1-k|^{\Omega(n'_j)}}{w(n'_j)}
  \\
\le & 2^{-\ell_j} \prod_{\substack{p\in P_j }} \Big( \sum_{i=0}^{\infty} \frac{\lambda^i(p)}{p^i} \frac{(2k)^{i}}{i!}\Big ( \sum_{l=0}^{i}
\frac{|1-k|^{l}}{l!}\lambda^l(p)\lambda(p^{i-l}) \Big ) \Big) \\
\le & 2^{-\ell_j} \prod_{\substack{p\in P_j }} \Big( 1+ \frac {2k(1+|1-k|)\lambda^2(p)}p+O(\frac 1{p^2})\Big) \\
\leq & 2^{-\ell_j/2}\prod_{\substack{p\in P_j }}\Big (1+ \frac {k^2\lambda^2(p)}p+O(\frac 1{p^2}) \Big ),
\end{split}
\end{align*}
  where the last estimation above follows by taking $N$ large enough and the
bound (which is a consequence of Lemma \ref{RS}) that
\begin{align}
\label{boundsforsumoverp}
   \frac 1{4N} \ell_j \leq \sum_{p \in P_j}\frac {\lambda^2(p)}{p} \leq \frac {2}N \ell_j.
\end{align}

  We then deduce from this, \eqref{6.02} and Lemma \ref{RS} that we have
\begin{align*}
& \sumstar_{\substack{ \chi \shortmod q }}L(\tfrac{1}{2},\chi) \mathcal{N}(\overline{\chi}, k) \mathcal{N}(\chi, k-1)
\gg \phis(q) \prod^R_{j=1}\Big (1+O( 2^{-\ell_j/2})\Big )\prod_{\substack{p\in P_j }}\Big (1+ \frac {k^2\lambda^2(p)}p+O(\frac 1{p^2}) \Big ) \gg
\phis(q)(\log q)^{k^2}.
\end{align*}
 This completes the proof of the proposition.

\section{Proof of Proposition \ref{Prop5}}
\label{sec 5}

  Observe that for a fixed integer $v$ such that $1 \leq v \leq R-1$, we have
\begin{align*}
\begin{split}
 & \sumstar_{\substack{ \chi \shortmod q }}|L(\tfrac{1}{2},f \otimes \chi)|^2 \Big (\prod^v_{j=1}|\mathcal{N}_j(\chi, k-1)|^2 \Big )|{\mathcal
 Q}_{v+1}(\chi, k)|^{2}
 \leq  \sum_{\substack{ \chi \shortmod q }}|L(\tfrac{1}{2},f \otimes \chi)|^2 \Big (\prod^v_{j=1}|\mathcal{N}_j(\chi, k-1)|^2 \Big )|{\mathcal
 Q}_{v+1}(\chi, k)|^{2} .
\end{split}
\end{align*}
    Since the sum over $e^{-\ell_j/2}$ converges, we deduce from the above that it remains to show that
\begin{align*}
\begin{split}
 \sum_{\substack{ \chi \shortmod q }}|L(\tfrac{1}{2},f \otimes \chi)|^2 \Big (\prod^v_{j=1}|\mathcal{N}_j(\chi, k-1)|^2 \Big )|{\mathcal
 Q}_{v+1}(\chi, k)|^{2}
 \ll & \phis(q)e^{-\ell_{v+1}/2}(\log q)^{k^2}.
\end{split}
\end{align*}

 Recall the definition of $c_k, r_k$ given in \eqref{r_k}. We further define the function $p_{v+1}(n)$ such that $p_{v+1}(n)=0$ or $1$, and that $p_{v+1}(n)=1$ if and only if $\Omega(n)=r_k\ell_{v+1}$ and all the prime factors of $n$ are from the interval $P_{v+1}$. Using these together with the notations in Section \ref{sec 4}, we see that
\begin{align}
\label{Pexpression}
  {\mathcal Q}_{v+1}(\chi, k)^{r_k\ell_{v+1}} =&  \Big( \frac{c_k  }{\ell_{v+1}}\Big)^{r_k\ell_{v+1}}\sum_{ \substack{ n_{v+1}}} \frac{\widetilde{\lambda}(n_{v+1})}{\sqrt{n_{v+1}}}\frac{(c_k\ell_{v+1})!
  }{w(n_{v+1})}\chi(n_{v+1})p_{v+1}(n_{v+1}).
\end{align}

  Note that $\prod^v_{j=1}\mathcal{N}_j(\chi, k-1){\mathcal
 Q}_{v+1}(\chi, k)$ is a short Dirichlet polynomial whose length does not exceed
$$q^{1/\ell_1+ \ldots +1/\ell_v+r_k/\ell_{v+1}} < q^{2r_k/10^{M}}.$$
  We apply \eqref{5.1}, \eqref{Pexpression} and the above observation to write for simplicity that
$$ \Big (\prod^v_{j=1}|\mathcal{N}_j(\chi, k-1)|^2 \Big )|{\mathcal
 Q}_{v+1}(\chi, k)|^{2} = \Big( \frac{c_k  }{\ell_{v+1}}\Big)^{2r_k\ell_{v+1}}((r_k\ell_{v+1})!)^2 \sum_{a,b \leq q^{2r_k/10^{M}}} \frac{u_a u_b}{\sqrt{ab}}\chi(a)\overline{\chi}(b).$$

   Similar to \eqref{xybounds}, we note that for all $a, b$,
\begin{align}
\label{ubounds}
 u_a, u_b \ll q^{\varepsilon}.
\end{align}

  We apply \eqref{lsquareapprox} to see that
\begin{align}
\label{LNsquaresum}
\begin{split}
& {\sum_{\chi \shortmod q}} |L(\half, f \otimes \chi)|^2\Big (\prod^v_{j=1}|\mathcal{N}_j(\chi, k-1)|^2 \Big )|{\mathcal
 Q}_{v+1}(\chi, k)|^{2} \\
=& 2\Big( \frac{c_k  }{\ell_{v+1}}\Big)^{2r_k\ell_{v+1}}((r_k\ell_{v+1})!)^2 \sum_{a,b \leq q^{2r_k/10^{M}}} \frac{u_a u_b}{\sqrt{ab}} \sum_{m, n} \frac{\lambda(m)\lambda(n)}{\sqrt{mn}} \Wf \left(\frac {mn}{q^2}\right)
{\sum_{\chi}} \chi(ma) \overline{\chi}(nb) \\
=& \phi(q) \Big( \frac {c_k }{\ell_{v+1}}\Big)^{2r_k\ell_{v+1}}((r_k\ell_{v+1})!)^2 \sum_{a,b \leq q^{2r_k/10^{M}}} \frac{u_a u_b}{\sqrt{ab}} \sum_{\substack{m,n \\ (mn, q)=1 \\ ma \equiv nb \,{\rm mod} \, q}}
\frac{\lambda(m)\lambda(n)}{\sqrt{mn}} \Wf \left(\frac {mn}{q^2}\right).
\end{split}
\end{align}

  We now consider the contribution of the terms $ma \neq nb$ in the last expression of \eqref{LNsquaresum}. Due to the rapid decay of $\Wf(x)$ given in \eqref{W}, we may assume that $mn \leq q^{2+\varepsilon}$.  We apply \cite[Lemma 1.6]{BFKMM} to see that there exist two non-negative function ${\mathcal V}_1(x), {\mathcal V}_2(x)$ supported on $[1/2,2]$,  satisfying
\begin{align}
\label{Vbounds}
\begin{split}
 {\mathcal V}^{(j)}_i(x) \ll_{j, \varepsilon} q^{j\varepsilon}.
\end{split}
\end{align}
  Moreover, we have the following smooth partition of unity:
\begin{align*}
\begin{split}
  \sum_{k \geq 0}{\mathcal V}_i(\frac x{2^k})=1, \quad i=1,2.
\end{split}
\end{align*}

  We also note the estimations
\begin{align}
\label{Stirling}
  (\frac ne)^n \leq n! \leq n(\frac ne)^n.
\end{align}

  It follows from this and the definition of $\ell_{v+1}$ given in Section \ref{sec 2'} that we have
\begin{align}
\label{lv1est}
 \Big( \frac {c_k }{\ell_{v+1}}\Big)^{2r_k\ell_{v+1}}((r_k\ell_{v+1})!)^2  \leq (r_k\ell_{v+1})^2 \Big( \frac{c_k r_k }{e } \Big)^{2r_k\ell_{v+1}} \ll q^{\varepsilon}.
\end{align}

  Applying the above, the definition of $\Wf (x)$ given in Lemma \ref{PropDirpoly} and the estimations given in \eqref{ubounds}, we see that the terms $ma \neq nb$ in the last expression of \eqref{LNsquaresum} contributes
\begin{align}
\label{offdiagbounds}
\begin{split}
 \ll & \frac {q^{1+\varepsilon}}{\sqrt{AB}} \sum_{a,b \leq q^{2r_k/10^{M}}}\frac {1}{\sqrt{ab}} \sum_{\substack{A=2^{k_1}, B=2^{k_2} \\ k_1, k_2 \geq 0 \\ AB \leq q^{2+\varepsilon}}} \int\limits_{(\varepsilon)} \Big |
\frac{\Gamma\left(\frac{\kappa}{2} + s \right)^2}{(2\pi )^s \Gamma\left(\frac{\kappa}{2} \right)^2 }\sum_{\substack{ma \neq nb \\ (mn, q)=1 \\ ma \equiv nb \,{\rm mod} \, q}}
\lambda(m)\lambda(n)V_1 \left(\frac {m}{A}\right) V_2 \left(\frac {n}{B}\right)\big (\frac {q}{mn}\big )^{s} \Big |
\frac{|ds|}{|s|},
\end{split}
\end{align}
 where
\begin{align*}
  V_i \left(x \right)=x^{-\half-s}{\mathcal V}_i(x), \quad i=1,2.
\end{align*}
  Due to the rapid decay of $\Gamma(s)$ on the vertical line, we may truncate the integral in \eqref{offdiagbounds} to $\Im(s) \leq (\log 5q)^2$ with a negligible error. This implies that the bounds given in \eqref{Vbounds} are also satisfied by $V_i (i=1,2)$ and that the expression in \eqref{offdiagbounds} can be further bounded by
\begin{align*}
\begin{split}
 \sum_{\substack{A=2^{k_1}, B=2^{k_2} \\ k_1, k_2 \geq 0 \\ AB \leq q^{2+\varepsilon}}}E(A,B),
\end{split}
\end{align*}
  where $A, B \geq 1$, $AB\leq q^{2+\varepsilon}$ and that
\begin{align*}
\begin{split}
& E(A, B)=\frac {q^{1+\varepsilon}}{\sqrt{AB}} \sum_{a,b \leq q^{2r_k/10^{M}}}\frac {1}{\sqrt{ab}}  \Big |\sum_{\substack{ma \neq nb \\ (mn, q)=1 \\ ma \equiv nb \,{\rm mod} \, q}}
\lambda(m)\lambda(n) V_1 \left(\frac {m}{A}\right) V_2 \left(\frac {n}{B}\right) \Big |.
\end{split}
\end{align*}

  We are thus led to estimate $E(A, B)$ for integers $A, B \geq 1, AB \leq q^{2+\varepsilon}$ and functions $V_1, V_2$ satisfying \eqref{Vbounds}.
In fact, this work has already been done in \cite[Section 12]{BM15} and
it follows from the result given in the first display below \cite[(12.7)]{BM15} that we have upon setting $X=q^{2r_k/10^{M}}, \theta=0$ there that $E(A,B) \ll q^{1-\varepsilon}$. This implies that  the contribution of the terms $ma \neq nb$ in the last expression of \eqref{LNsquaresum} is negligible.

 It remains to consider the terms $ma = nb$ in the last expression of \eqref{LNsquaresum}. We write $m = \frac{\alpha b}{(a,b)}, n =
 \frac{\alpha a}{(a,b)}$ and apply the estimation in \eqref{lv1est} to see that these terms are
\begin{align}
\label{mainterm}
\ll & \phi(q)(r_k\ell_{v+1})^2 \Big( \frac{c_k r_k }{e } \Big)^{2r_k\ell_{v+1}} \sum_{a,b \leq q^{2r_k/10^{M}}} \frac{(a,b)}{ab} u_a u_b  \sum_{(\alpha, q)=1} \frac{\lambda(\frac{\alpha b}{(a,b)})\lambda(\frac{\alpha a}{(a,b)})}{\alpha} \Wf \left(\frac{\alpha^2
ab}{q^2(a,b)^2}\right).
\end{align}

 To evaluate the last sum above, we first recall that the Rankin-Selberg $L$-function of $L(s, f \times f)$ of $f$ is defined for $\Re(s)>1$ by (see \cite[(23.24)]{iwakow})
\begin{align*}
 L(s, f \times f)=\sum_{n \geq 1}\frac {\lambda^2(n)}{n^s}.
\end{align*}
  It is known (see \cite[p. 132]{iwakow}) that $L(s, f \times f)$ has a simple at $s=1$. In fact, we have (see \cite[(5.97)]{iwakow})
\begin{align*}
  L(s, f \times f)=\zeta(s)L(s, \operatorname{sym}^2 f),
\end{align*}
  where $L(s, \operatorname{sym}^2 f)$ is the symmetric square $L$-function of $f$ defined for $\Re(s)>1$ by (see \cite[(25.73)]{iwakow})
\begin{align*}
 L(s, \operatorname{sym}^2 f)=& \zeta(2s) \sum_{n \geq 1}\frac {\lambda(n^2)}{n^s}=\prod_{p}(1-\frac {\lambda(p^2)}{p^s}+\frac {\lambda(p^2)}{p^{2s}}-\frac {1}{p^{3s}} )^{-1}.
\end{align*}
  It follows from a result of G. Shimura \cite{Shimura} that the corresponding completed $L$-function
\begin{align*}
 \Lambda(s, \operatorname{sym}^2 f)=& \pi^{-3s/2}\Gamma (\frac {s+1}{2})\Gamma (\frac {s+\kappa-1}{2}) \Gamma (\frac {s+\kappa}{2}) L(s, \operatorname{sym}^2 f).
\end{align*}
  is entire and satisfies the functional equation
$\Lambda(s, \operatorname{sym}^2 f)=\Lambda(1-s, \operatorname{sym}^2 f)$.
 Combining this with \cite[(5.8)]{iwakow} and apply the convexity bounds (see \cite[Exercise 3, p.
  100]{iwakow}) for $L$-functions, we deduce that
\begin{align}
\label{Lsymest}
\begin{split}
  L(s, \operatorname{sym}^2 f) \ll & \left( 1+|s| \right)^{\frac {3(1-\Re(s))}{2}+\varepsilon}, \quad 0 \leq \Re(s) \leq 1.
\end{split}
\end{align}

  Note also the following convexity bound for $\zeta(s)$:
\begin{align}
\label{zetaest}
\begin{split}
  \zeta(s) \ll & \left( 1+|s| \right)^{\frac {1-\Re(s)}{2}+\varepsilon}, \quad 0 \leq \Re(s) \leq 1.
\end{split}
\end{align}

 We now evaluate the last sum in \eqref{mainterm} by setting $X=q^2(a,b)^2/(ab)$ there and applying the definition of $\Wf (x)$ given in Lemma \ref{PropDirpoly} to obtain that
\begin{align}
\label{maintermint}
& \sum_{(\alpha, q)=1} \frac{\lambda(\frac{\alpha b}{(a,b)})\lambda(\frac{\alpha a}{(a,b)})}{\alpha}  \Wf \left(\frac{\alpha^2}{X}\right)= \frac{1}{2\pi i} \int\limits_{(c)}
\frac{\Gamma\left(\frac{\kappa}{2} + s \right)^2}{(2\pi )^s \Gamma\left(\frac{\kappa}{2} \right)^2 }L(1+2s, f \times f) H(1+2s; q)G(1+2s;a,b)X^{s}
\frac{ds}{s},
\end{align}
  where $H(s;q)=\prod_{p|q}(1-\frac {\lambda^2(p)}{p^{1+2s}}+ O(\frac {1}{p^{2(1+2s)}}))$,  $G(s; a, b)=\prod_{p}G_p(s;a,b)$ with
\begin{equation*}
\begin{split}
G_{p}(s;a,b) =\left\{
 \begin{array}
  [c]{ll}
  \lambda(p^l)+ O(\frac {1}{p^{1+2s}}) & \text{if } p|  ab/(a,b)^2,\\
  1  &\text{otherwise }.
 \end{array}
 \right.
 \end{split}
\end{equation*}

  We evaluate the integral in \eqref{maintermint} by shifting the line of integration to $\Re(s)=-1/4+\varepsilon$. We encounter a double pole at $s=0$ in the
  process. Note that on the new line, we have for some constant $B_1$,
\begin{align*}
\begin{split}
 G(-\frac 14 +\varepsilon; a,b) \ll B_1^{\Omega(q)+\Omega(\frac {a}{(a,b)}+\Omega(\frac {b}{(a,b)})}d\Big(\frac {a}{(a,b)}\Big)d\Big(\frac {b}{(a,b)}\Big) \ll q^{\varepsilon},
\end{split}
\end{align*}
  where the last estimation above can be obtained using arguments that lead to \eqref{xybounds}.

  Combining the above with \eqref{Lsymest}, \eqref{zetaest} and the rapid decay of $\Gamma(s)$ when $|\Im(s)| \rightarrow \infty$, we deduce that the integration on the new line is
\begin{align*}
\begin{split}
  \ll q^{-\half+\varepsilon}.
\end{split}
\end{align*}
  Applying this in \eqref{mainterm} and taking note of the  definition of $\ell_{v+1}$ given in Section \ref{sec 2'}, we see that the contribution of the integration on the new line to the right side of \eqref{mainterm} is
\begin{align*}
\begin{split}
  \ll \phi(q)(r_k\ell_{v+1})^2 \Big( \frac{c_k r_k }{e } \Big)^{2r_k\ell_{v+1}} q^{4r_k/10^{M}-\half+\varepsilon} \ll q^{\half+4r_k/10^{M}+\varepsilon} \ll q^{1-\varepsilon}.
\end{split}
\end{align*}

   We now evaluate the corresponding residue to see that
\begin{align}
\label{alphasum}
& \sum_{(\alpha, q)=1} \frac{\lambda(\frac{\alpha b}{(a,b)})\lambda(\frac{\alpha a}{(a,b)})}{\alpha}  \Wf \left(\frac{\alpha^2}{X}\right)= C_1(q)L(1, \operatorname{sym}^2 f)G(0;a,b) \Big (\log X+2\sum_{p|ab/(a,b)^2}\frac {G'_p(1;a,b)}{G_p(1;a,b)}+C_2(q) \Big )+O( q^{-\half+\varepsilon}),
\end{align}
   where $C_1(q), C_2(q)$ are some constants depending on $q$ only, satisfying $C_i(q) \ll 1, i=1,2$.

  We apply \eqref{alphasum} to evaluate \eqref{mainterm} to see that we may ignore the contribution of the error term in \eqref{alphasum} so
  that the expression in \eqref{mainterm} is
\begin{align*}
 \ll & \phi(q) (r_k\ell_{v+1})^2 \Big( \frac{c_k r_k }{e } \Big)^{2r_k\ell_{v+1}} \\
& \times \sum_{a,b \leq q^{2r_k/10^{M}}} \frac{(a,b)}{ab} u_a u_b  G(0;a,b)\Big ( 2\log q+2 \log (a, b)-\log a -\log b+2\sum_{p|ab/(a,b)^2}\frac {G'_p(0;a,b)}{G_p(0;a,b)}+C_2(q) \Big ).
\end{align*}

  As the estimations are similar, it suffices to treat the sum
\begin{align}
\label{sumoverlog}
\begin{split}
& \phi(q)(r_k\ell_{v+1})^2 \Big( \frac{c_k r_k }{e } \Big)^{2r_k\ell_{v+1}} \sum_{a,b } \frac{(a,b)G(0;a,b)}{ab} u_a u_b \log a \\
=& \phi(q)(r_k\ell_{v+1})^2 \Big( \frac{c_k r_k }{e } \Big)^{2r_k\ell_{v+1}} \sum_{p \in \bigcup^{v+1}_{j=1} P_j}\sum_{l_1 \geq 1, l_2 \geq 0}\frac {l_1 \log p}{p^{l_1+l_2-\min (l_1, l_2)}}\sum_{\substack{ a,b \\ (ab, p)=1} } \frac{(a,b)G(0;p^{l_1}a,p^{l_2}b) u_{p^{l_1}a} u_{p^{l_2}b}}{ab} .
\end{split}
\end{align}

   We now consider the last sum above for fixed $p=p_1, l_1, l_2$. Without loss of generality, we may assume that $p_1 \in P_1$. We then define for $(n_1n'_1, p_1)=1$,
\begin{equation*}
 v_{n_1} = \frac{1}{n_1} \frac{(k-1)^{\Omega(n_1)}\widetilde{\lambda}(n_{1})}{w(n_1)}  b_1(n_1p_1^{l_1}), \quad v_{n'_1} = \frac{1}{n'_1} \frac{(k-1)^{\Omega(n'_1)}\widetilde{\lambda}(n'_{1})}{w(n'_1)}  b_1(n'_1p_1^{l_2}).
\end{equation*}
  For $2 \leq j \leq v$,
\begin{equation*}
 v_{n_j} = \frac{1}{n_j} \frac{(k-1)^{\Omega(n_j)}\widetilde{\lambda}(n_{j})}{w(n_j)}  b_j(n_j), \quad  v_{n'_j} = \frac{1}{n'_j} \frac{(k-1)^{\Omega(n'_j)}\widetilde{\lambda}(n'_{j})}{w(n'_j)}  b_j(n'_j).
\end{equation*}
  Also,
\begin{equation*}
 v_{n_{v+1}} = \frac{1}{n_{v+1}} \frac{\widetilde{\lambda}(n_{v+1})}{w(n_{v+1})}  p_{v+1}(n_{v+1}), \quad  v_{n'_{v+1}} = \frac{1}{n'_{v+1}} \frac{\widetilde{\lambda}(n'_{v+1})}{w(n'_{v+1})}  p_{v+1}(n'_{v+1}).
\end{equation*}

  Then one checks that
\begin{align}
\label{doublesum}
 \sum_{\substack{ a,b \\ (ab, p_1)=1} } \frac{(a,b)G(0;p^{l_1}a,p^{l_2}b) u_{p^{l_1}a} u_{p^{l_2}b}}{ab}=\frac {(k-1)^{l_1+l_2}\lambda^{l_1+l_2}(p_{1})G(0;p^{l_1},p^{l_2})}{l_1!l_2!}\prod^{v+1}_{j=1}\Big ( \sum_{\substack{n_j, n'_j \\ (n_1n'_1, p_1)=1}}(n_j, n'_j)G(0;n_j,n'_j)v_{n_j}v_{n'_j} \Big ).
\end{align}

   As in the proof of Proposition \ref{Prop4}, we remove the restriction of $b_1(n_1p_1^{l_1})$ on $\Omega(n_1p_1^{l_1})$
   and $b_1(n'_1p_1^{l_2})$ on $\Omega(n'_1p_1^{l_2})$ to see that the sum on the right side in \eqref{doublesum} for $j = 1$  becomes
\begin{align*}
  & \sum_{\substack{n_1, n'_1\\ (n_1n'_1, p_1)=1}}(n_1, n'_1)G(0;n_1,n'_1)\frac{\widetilde{\lambda}(n_{1})}{n_1} \frac{(k-1)^{\Omega(n_1)}}{w(n_1)} \frac{\widetilde{\lambda}(n'_{1})}{n'_1} \frac{(k-1)^{\Omega(n'_1)}}{w(n'_1)} =
  \prod_{\substack{p \in P_1 \\ p \neq p_1}}\Big ( 1+\frac{(2(k-1)+(k-1)^2)\lambda^2(p)}{p} +O(\frac 1{p^2})  \Big ) \\
\ll &  \exp (\sum_{p \in P_1}\frac{(k^2-1)\lambda^2(p)}{p}+O(\sum_{p \in P_1}\frac 1{p^2})).
\end{align*}

  Further, we notice that in this case we have $2^{\Omega(n)+l_i-\ell_1}\ge 1$ for $i=1,2$.  Thus, we apply Rankin's
  trick to see that the error introduced this way is
\begin{align*}
 \ll &  \big(2^{l_1}+2^{l_2} \big ) 2^{-\ell_1} \sum_{n_1, n'_1}\frac {(n_1, n'_1)|G(0;n_1,n'_1)|}{n_1n'_1} \frac{|k-1|^{\Omega(n_1)}|\widetilde{\lambda}(n_{1})|}{w(n_1)} \frac{|k-1|^{\Omega(n'_1)}2^{\Omega(n'_1)}|\widetilde{\lambda}(n'_{1})|}{w(n'_1)}.
\end{align*}

   We then deduce that
\begin{align*}
 &  \sum_{l_1 \geq 1, l_2 \geq 0} \frac {l_1 \log p_1}{p^{l_1+l_2-\min (l_1, l_2)}_1}\frac {(k-1)^{l_1+l_2}}{l_1!l_2!}\sum_{\substack{n_1, n'_1
 \\ (n_1n'_1, p_1)=1}}(n_1, n'_1)G(0;n_1,n'_1)v_{n_1}v_{n'_1} \\
 \ll & \sum_{l_1 \geq 1, l_2 \geq 0} \frac {l_1 \log p_1}{p^{l_1+l_2-\min (l_1, l_2)}_1} \frac {(k-1)^{l_1+l_2}}{l_1!l_2!} \Big (1+O \big ( \big (2^{l_1}+2^{l_2} \big ) 2^{-\ell_1/2} \big ) \Big )  \exp (\sum_{p \in P_1}\frac{(k^2-1)\lambda^2(p)}{p}+O(\sum_{p \in P_1}\frac 1{p^2})) \\
 \ll & \Big (1+O \big (  2^{-\ell_1/2} \big )\Big ) \Big ( \frac{\log p_1}{p_1} + O(\frac
 {\log p_1}{p^2_1}) \Big )   \exp (\sum_{p \in  P_1}\frac{(k^2-1)\lambda^2(p)}{p}+O(\sum_{p \in  P_1}\frac 1{p^2})).
\end{align*}

   Similar estimations carry over to the sums over $n_j, n'_j$ for $2 \leq j \leq v$ in \eqref{doublesum}. To treat the sum over $n_{v+1}, n'_{v+1}$,  we apply Rankin's trick again to see that the sum is
\begin{align*}
\ \ll &  (c_kr_k)^{-2r_k\ell_{v+1}} \sum_{\substack{n_{v+1}, n'_{v+1} \\ p |
n_{v+1}n'_{v+1} \implies  p\in P_{v+1} } }\frac {(n_{v+1}, n'_{v+1})|\widetilde{\lambda}(n_{v+1})\widetilde{\lambda}(n'_{v+1})G(0;n_{v+1},n'_{v+1})|}{n_{v+1}n'_{v+1}} \frac{(c_kr_k)^{\Omega(n_{v+1})}}{w(n_{v+1})} \frac{(c_kr_k)^{\Omega(n'_{v+1})}}{w(n'_{v+1})}.
\end{align*}
  By taking $N$ large enough, we deduce from this that
\begin{align*}
\begin{split}
&  (r_k\ell_{v+1})^2 \Big( \frac{c_k r_k }{e } \Big)^{2r_k\ell_{v+1}} \sum_{\substack{n_{v+1}, n'_{v+1} \\ (n_{v+1}n'_{v+1}, p_1)=1}}(n_{v+1}, n'_{v+1})v_{n_{v+1}}v_{n'_{v+1}}
\ll  e^{-\ell_{v+1}}\exp (\sum_{p \in P_{v+1}}\frac{(k^2-1)\lambda^2(p)}{p}+O(\sum_{p \in P_{v+1}}\frac 1{p^2})).
\end{split}
\end{align*}

   It follows from the above discussions that we have
\begin{align}
\label{sumpbound}
\begin{split}
&  (r_k\ell_{v+1})^2 \Big( \frac{c_k r_k }{e } \Big)^{2r_k\ell_{v+1}} \sum_{l_1 \geq 1, l_2 \geq 0}\frac {l_1 \log p_1}{p^{l_1+l_2-\min (l_1, l_2)}_1}\frac {(k-1)^{l_1+l_2}}{l_1!l_2!}\sum_{\substack{ a,b \\ (ab, p_1)=1} } \frac{u_{p^{l_1}_1a}
u_{p^{l_2}_1b}}{ab} \\
 \ll & e^{-\ell_{v+1}}\prod^v_{j=1}\big(1+O(2^{-\ell_j/2})\big)\exp (\sum_{p \in \bigcup^{v+1}_{j=1} P_j}\frac{(k^2-1)\lambda^2(p)}{p}+O(\sum_{p \in \bigcup^{v+1}_{j=1} P_j}\frac 1{p^2})) \times \Big ( \frac{\log p_1}{p_1} + O(\frac
 {\log p_1}{p^2_1}) \Big ).
\end{split}
\end{align}

  We now apply \eqref{boundsforsumoverp} and the observation $\ell_j > \ell^2_{j+1}>2\ell_{j+1}$ to see that
\begin{align}
\label{sumpbound1}
 \sum_{p \in \bigcup^R_{j=v+2} P_j}\frac{|k^2-1|\lambda^2(p)}{p} \leq \sum^{R}_{j=v+2}\sum_{p \in P_j}\frac{|k^2-1|\lambda^2(p)}{p} \leq \frac {2|k^2-1|}{N} \sum^{R}_{j=v+2} \ell_j \leq \frac {2|k^2-1|\ell_{v+2}}{N} \sum^{\infty}_{j=0} \frac 1{2^j} \leq \frac {2|k^2-1|\ell_{v+1}}{N}.
\end{align}

   It follows from \eqref{sumpbound1} that the last expression in \eqref{sumpbound} is
\begin{align*}
\begin{split}
 \ll &  e^{-\ell_{v+1}/2}\exp (\sum_{p \in \bigcup^{R}_{j=1} P_j}\frac{(k^2-1)\lambda^2(p)}{p}+O(\sum_{p \in \bigcup^{R}_{j=1} P_j}\frac 1{p^2})) \times \Big ( \frac{\log p_1}{p_1} + O(\frac
 {\log p_1}{p^2_1}) \Big ).
\end{split}
\end{align*}

 We then conclude from the above, \eqref{sumoverlog}, Lemma \ref{RS} and the observation that $\phi(q) \ll \phis(q)$ that
\begin{align*}
\begin{split}
& \sum_{\substack{ \chi \shortmod q }}|L(\tfrac{1}{2},f \otimes \chi)|^2 \Big (\prod^v_{j=1}|\mathcal{N}_j(\chi, k-1)|^2 \Big )|{\mathcal
 Q}_{v+1}(\chi, k)|^{2} \\
\ll &    \phi(q) e^{-\ell_{v+1}/2}\exp (\sum_{p \in \bigcup^{R}_{j=1} P_j}\frac{(k^2-1)\lambda^2(p)}{p}+O(\sum_{p \in \bigcup^{R}_{j=1} P_j}\frac 1{p^2})) \times \sum_{p \in \bigcup^{v+1}_{j=1}P_j} \Big ( \frac{\log p}{p} + O(\frac
 {\log p}{p^2}) \Big ) \\
\ll & \phis(q) e^{-\ell_{v+1}/2}(\log q)^{k^2}.
\end{split}
\end{align*}
  This completes the proof of the proposition.

\section{Proof of Proposition \ref{Prop6}}

   As the proofs are similar, we shall only prove here that
\begin{align*}
  \sumstar_{\substack{ \chi \shortmod q }}\prod^R_{j=1}\big ( |{\mathcal N}_j(\chi, k)|^2+ |{\mathcal Q}_j(\chi,k)|^2 \big ) \ll \phis(q)(\log q)^{ k^2 }.
\end{align*}

   We first note that
\begin{align}
\label{upperboundprodofN}
\begin{split}
 & \sumstar_{\substack{ \chi \shortmod q }}\prod^R_{j=1}\big ( |{\mathcal N}_j(\chi, k)|^2+ |{\mathcal Q}_j(\chi,k)|^2 \big )
\leq  \sum_{\substack{ \chi \shortmod q }}  \prod^R_{j=1}\big ( |{\mathcal N}_j(\chi, k)|^2+|{\mathcal Q}_j(\chi,k)|^2 \big ).
\end{split}
\end{align}

  We shall take $M$ large enough so that we may deduce from \eqref{sumoverell} that
\begin{align*}
  (2 r_k+2) \sum^R_{j=1}\frac 1{\ell_j} \leq \frac {4(r_k+1)}{\ell_R} <1 .
\end{align*}

  We then apply \eqref{5.1} and \eqref{Pexpression} in the last sum of \eqref{upperboundprodofN} and deduce from the above that
  the orthogonality relation for characters modulo $q$ implies that only the
  diagonal terms in the last sum of \eqref{upperboundprodofN} survive. Thus we obtain that
\begin{align}
\label{maintermbound}
\begin{split}
 & \sum_{\substack{ \chi \shortmod q }}  \prod^R_{j=1}\big ( |{\mathcal N}_j(\chi, k)|^2+ |{\mathcal Q}_j(\chi,k)|^2 \big )
\le  \phi(q) \prod^R_{j=1} \Big (  \sum_{n_j} \frac{k^{2\Omega(n_j)}\widetilde{\lambda}^2(n_j)}{n_j w^2(n_j)}  b_j(n_j)  + \Big( \frac{c_k} {\ell_j}\Big)^{2r_k\ell_j}((r_k\ell_j)!)^2 \sum_{ \substack{ \Omega(n_j) = r_k\ell_j \\ p|n_j \implies  p\in P_j}} \frac{\widetilde{\lambda}^2(n_j) }{n_j w^2(n_j)} \Big ).
\end{split}
\end{align}

  Arguing as before, we see that
\begin{align}
\label{sqinN}
\begin{split}
  \sum_{n_j} \frac{k^{2\Omega(n_j)}\widetilde{\lambda}^2(n_j)}{n_j w^2(n_j)}  b_j(n_j)=\Big(1+ O\big(2^{-\ell_j/2} \big ) \Big)
  \exp (\sum_{p \in P_j} \frac {k^2\lambda^2(p)}{p}+ O(\sum_{p \in P_j} \frac {1}{p^2})).
\end{split}
\end{align}
   Note also that, as $w^2(n) \geq w(n)$ and $\widetilde{\lambda}^2 (n) \geq 0$, we have
\begin{align*}
\begin{split}
 \sum_{ \substack{ \Omega(n_j) = r_k\ell_j \\ p|n_j \implies  p\in P_j}} \frac{\widetilde{\lambda}^2 (n_j)}{n_j w^2(n_j)}  \leq  \frac 1{(r_k \ell_j)!} \Big (\sum_{p \in P_j}\frac {\lambda^2(p)}{p} \Big )^{r_k\ell_j}.
\end{split}
\end{align*}

  Now, we apply the estimations given in \eqref{boundsforsumoverp} and \eqref{Stirling}  to deduce from the above that by taking $M, N$ large enough,
\begin{align}
\label{Qest}
\begin{split}
 &  \Big( \frac{c_k} {\ell_j}\Big)^{2r_k\ell_j}((r_k\ell_j)!)^2 \sum_{ \substack{ \Omega(n_j) = r_k\ell_j \\ p|n_j \implies  p\in P_j}} \frac{\widetilde{\lambda}^2(n_j) }{n_j w^2(n_j)} \ll r_k\ell_j\Big( \frac{c^2_k r_k }{e\ell_j} \Big)^{r_k\ell_j}\Big (\sum_{p \in P_j}\frac {\lambda^2(p)}{p} \Big )^{r_k\ell_j} \\
\ll &  r_k\ell_j\Big( \frac{c^2_k r_k }{e\ell_j} \Big)^{r_k\ell_j}e^{r_k\ell_j \log (2\ell_j/N)}
\ll  e^{-\ell_j} \exp (\sum_{p \in P_j} \frac {k^2\lambda^2(p)}{p}+ O(\sum_{p \in P_j} \frac {1}{p^2})).
\end{split}
\end{align}
  Applying \eqref{sqinN} and \eqref{Qest} in \eqref{maintermbound}, together with Lemma \ref{RS} and the observation that $\phi(q) \ll \phis(q)$,  we readily deduce the assertion of the proposition.

\vspace*{.5cm}

\noindent{\bf Acknowledgments.} P.G. is supported in part by NSFC grant 11871082, X. He is supported in part by NSFC grant 12101427 and X. Wu is supported in part by NSFC grant 11871187.

\bibliography{biblio}
\bibliographystyle{amsxport}

\vspace*{.5cm}

\vspace*{.5cm}

\noindent\begin{tabular}{p{6cm}p{6cm}p{6cm}}
School of Mathematical Sciences & College of Mathematical Sciences & School of Mathematics   \\
Beihang University  &    Sichuan University & Hefei University of Technology \\
Beijing 100191 China    & Chengdu Sichuan 610016 China & Hefei Anhui 230009 China \\
Email: {\tt penggao@buaa.edu.cn} & Email: {\tt hexiaoguangsdu@gmail.com} & Email: {\tt xswu@amss.ac.cn} \\
\end{tabular}

\end{document}